\documentclass [12pt,oneside]{amsart}
\usepackage{amsmath}

\usepackage {amsfonts}
\usepackage {amsmath}
\usepackage {amsthm}
\usepackage {amssymb}
\usepackage {framed}
\usepackage {amsxtra}
\usepackage {enumerate}
\usepackage {graphicx}
\usepackage{color}
\usepackage{graphicx}
\usepackage[left=0.87in, right=.86 in, top=1.2 in, footskip=0.5 in, bottom=0.9 in]{geometry}
\usepackage[font=small,skip=0pt]{caption}

\makeatletter

\theoremstyle{definition}
\newtheorem{df}{Definition} [section]

\theoremstyle{plain}
\newtheorem{thm}[df]{Theorem}

\newtheorem{lemma}[df]{Lemma}

\newtheorem{obs}[df]{Observation}

\title{On the polygon determined by the short diagonals of a convex polygon}

\author{}

\author{Jacqueline Cho$^1$}
\thanks{$^1$Phillips Exeter Academy, Exeter, NH 03833, \texttt{jcho1@exeter.edu}}

\author{Dan Ismailescu$^2$}
\thanks{$^2$Mathematics Department, Hofstra University, Hempstead, NY 11549, \texttt{dan.p.ismailescu@hofstra.edu}}

\author{Yiwon Kim $^3$}
\thanks{$^3$Taft School, Watertown, CT 06795, \texttt{stefankim@taftschool.edu}}

\author{Andrew Woojong Lee$^4$}
\thanks{$^4$Choate Rosemary Hall, Wallingford, CT 06492, \texttt{alee21@choate.edu}}

\begin{document}

\begin{abstract}
Let $K$ be a convex pentagon in the plane and let $K_1$ be the pentagon bounded by the diagonals of $K$.
It has been conjectured that the maximum of the ratio between the areas of $K_1$ and $K$ is reached when $K$ is an affine regular pentagon.
In this paper we prove this conjecture. We also show that for polygons with at least six vertices the trivial answers are the best possible.
\end{abstract}

\maketitle

\pagenumbering{arabic}


\begin{section}{\bf The Problem}
Given $K=A_1A_2\ldots A_n$, a convex polygon with $n$ vertices with $n\ge 5$, we draw the short diagonals $A_1A_3,\,A_2A_4,$ $\ldots, A_{n-2}A_n,\,A_{n-1}A_1,\,A_nA_2$. A new convex $n$-gon, $K_1=B_1B_2\ldots B_n$ is created inside $K$, where
$B_i$ is the intersection point of $A_{i-1}A_{i+1}$ and $A_iA_{i+2}$ for all $1\le i\le n$ as shown in figure \ref{l1}.
\vspace{-1cm}

\begin{figure}[!htb]
\centering
\includegraphics[scale=1.7]{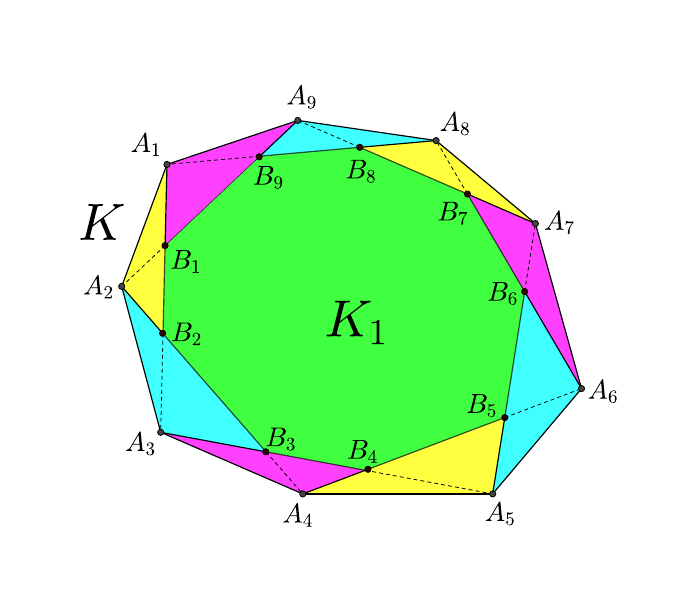}
\vspace{-1cm}
\caption{The polygon $K_1$ created by the short diagonals of $K$}
\label{l1}
\end{figure}
Obviously, for every $K$, the ratio between the area of $K_1$ and the area of $K$ lies in the interval $(0,1)$, but can we do better? The somewhat surprising answer is that, with one notable exception, the answer is no!

Let us first introduce several notations. The area of any polygon will be denoted by the symbol $\Delta$. That is, the area of the polygon $K$ is denoted as $\Delta(K)$ or $\Delta(A_1A_2\ldots A_n)$. We make the convention to list the vertices of any simple polygon in counterclockwise order.

A {\it peripheral} triangle of $K$ is a triangle formed by three consecutive vertices of $K$. In other words every triangle of the form $A_{k}A_{k+1}A_{k+2}$ with $1\le k \le n$ is such a triangle. A {\it marginal} triangle of $K$ is any triangle of the form $A_kA_{k+1}B_{k+1}$; these triangles are colored in figure \ref{l1}.

We denote the sum of the areas of the peripheral triangles by $\Omega(K)$, and the sum of the areas of all the marginal triangles by $\Psi(K)$ that is,
\begin{equation}\label{periphery}
\Omega(K)=\sum_{k=1}^n \Delta(A_kA_{k+1}A_{k+2}),\quad\text{and}\quad \Phi(K)=\sum_{k=1}^n \Delta(A_kA_{k+1}B_{k+1})
\end{equation}
Since clearly $\Delta(A_kA_{k+1}A_{k+2})> \Delta(A_kA_{k+1}B_{k+1})$ it follows that
\begin{equation}\label{K1K}
\Delta(K_1)=\Delta(K)-\Phi(K)>\Delta(K)-\Omega(K).
\end{equation}

We can now prove our first result.
\begin{thm}\label{easythm}
(a) For every integer $n\ge 5$ and every $\epsilon>0$ there exists a convex $n$-gon $K$ such that $\Delta(K_1)/\Delta(K)<\epsilon$.

(b) For every integer $n\ge 6$ and every $\epsilon>0$ there exists a convex $n$-gon $K$ such that $\Delta(K_1)/\Delta(K)>1-\epsilon$.
\end{thm}

\begin{figure}[!htb]
\centering
\includegraphics[scale=1.3]{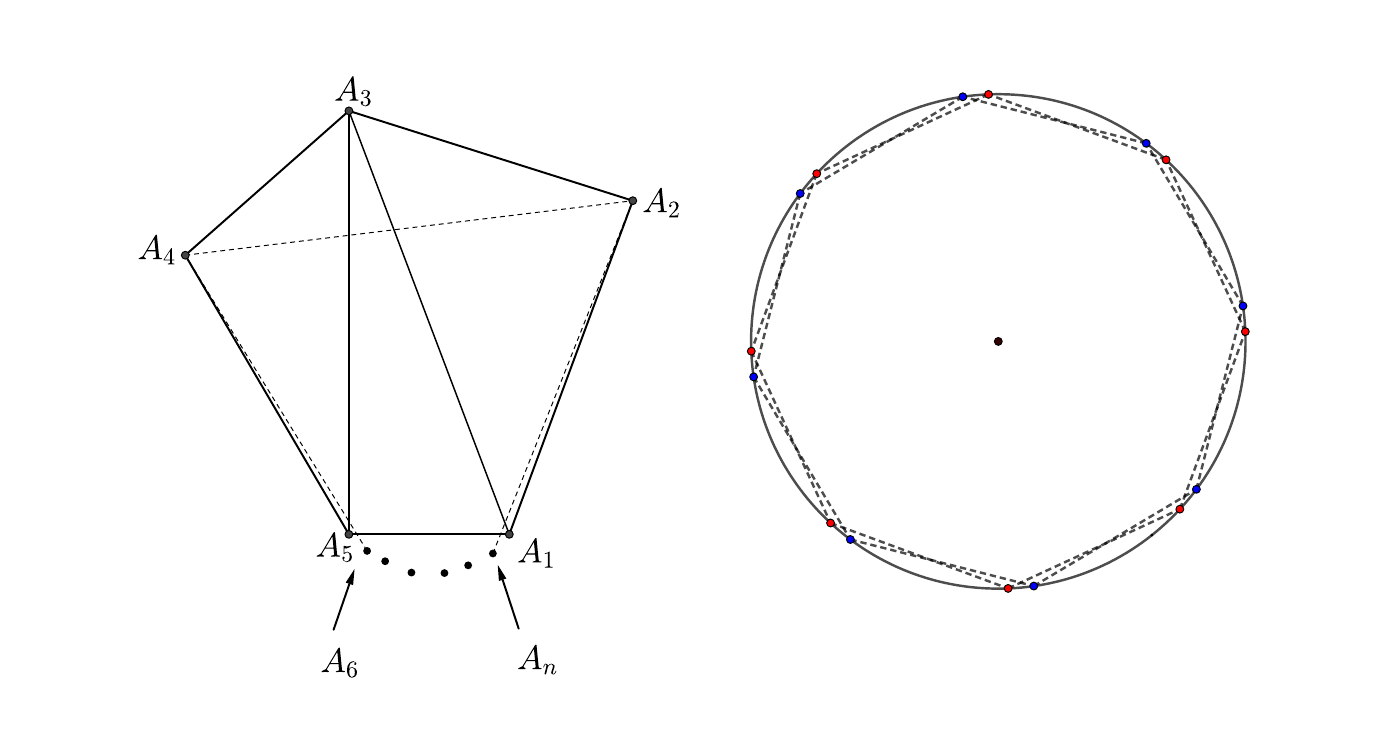}
\vspace{0cm}
\caption{ (a) Construction for $\Delta(K_1)/\Delta(K)<\epsilon$ (b) Construction for $\Delta(K_1)/\Delta(K)>1-\epsilon$}
\label{l2}
\end{figure}

\begin{proof}
For the first part we construct $K$ as follows.
Select first the vertices $A_1$, $A_3$ and $A_5$ so that $\Delta(A_1A_3A_5)=\epsilon$. Next, place the vertices
$A_2$ and $A_4$ as in figure \ref{l2}$(a)$ with $\Delta(A_1A_2A_3)=\Delta(A_1A_5A_4)=1$. Finally, place the remaining
vertices (if any) along a flat arc close to the segment $A_1A_5$ so that the convex hull of the points $A_1,\,A_5,\,A_6,\ldots,\,A_n$ has area smaller than $\epsilon$.
It is easy to see that the polygon $K_1$ lies within the boundaries of the convex polygon
obtained by removing vertices $A_2$ and $A_4$ from $K$. It follows that $\Delta(K_1)\le 2\epsilon$.
On the other hand, by construction $\Delta(K)\ge 2$. It follows that $\Delta(K_1)/\Delta(K)\le \epsilon$ as claimed.

Next, we proceed to proving the second part. Let $0<\epsilon$ and assume first that $n=2m$ is even.
Recall that $n\ge 6$ and therefore $m\ge 3$.
Start by considering a regular $m$-gon inscribed in a circle of unit radius (the red points), then rotate it by an
angle of $\epsilon/20$ to obtain the blue points, as in figure \ref{l2}$(b)$.
Let $K$ be the convex $n$-gon whose vertices are the blue and the red points.

One can bound the area of every peripheral triangle of $K$ as follows
\begin{equation}
\Delta(A_{k}A_{k+1}A_{k+2})\le \frac{1}{2}\cdot \frac{\epsilon}{20}\cdot\frac{2\pi}{m}\cdot\sin\frac{\pi}{2m}\le \frac{\epsilon}{n^2},
\end{equation}
and therefore $\Omega(K)\le \epsilon/n$. Using now \eqref{K1K} we have that
\begin{equation*}
\Delta(K_1)>\Delta(K)-\Omega(K)\ge \Delta(K)-\frac{\epsilon}{n}>(1-\epsilon)\Delta(K),
\end{equation*}
where the last inequality is equivalent to $\Delta(K)>1/n$ which is obviously true.
For $n=2m+1$ odd, the same construction works; place the additional vertex close to one of the existing $m$ pairs of points.
Since neither the area of $K$, nor the area of $K_1$ change by much, the conclusion remains valid.
\end{proof}

The theorem above solves the problem in all cases but one: {\it how large can the ratio $\Delta(K_1)/\Delta(K)$ be when $K$ is a pentagon?} This innocent looking problem remained unsolved despite the fact that quite a few established researchers expressed interest in solving it.

In this paper we answer this question in the following
\begin{thm}\label{main}
Let $K$ be a convex pentagon and let $K_1$ be the pentagon formed by the diagonals of $K$. Then
\begin{equation*}
\frac{\Delta(K_1)}{\Delta(K)}\le \frac{7-3\sqrt{5}}{2}=0.14858\ldots,
\end{equation*}
with equality if and only if $K$ is an affine regular pentagon.
\end{thm}
\end{section}

\begin{section}{\bf Prior attempts}
The result in Theorem \ref{main} ``feels'' true and was conjectured in \cite{elkies} where several contributors suggested
various lines of attack. Below we describe three such strategies.

The first approach is via cartesian coordinates: this looks quite tempting since one can fix three vertices of the pentagon
at say $(0,0)$, $(1,0)$ and $(0,1)$. This simplifying assumption is possible due to the fact that the problem
is affine invariant. The problem becomes an optimization in four variables - the coordinates of the remaining two vertices.
The drawback is that since one has to take convexity into account, there are several side inequalities that make the analysis of the resulting optimization problem  rather difficult.

The second technique assumes that the vertices of the initial pentagon lie on a circle. Indeed, the five vertices of $K$  determine a unique ellipse, and after an appropriate affine transformation we may assume that this ellipse is in fact a circle.

This technique handles nicely the convexity restrictions but then one can fix the coordinates of only one vertex of $K$, say $A(1,0)$, while the coordinates of the remaining four points will be of the form $(\cos{t_i},\sin{t_i})$, for $0<t_1<t_2<t_3<t_4<2\pi$.

In the end, we are still left with four variables, the side restriction $0<t_1<t_2<t_3<t_4<2\pi$ and a bunch of trigonometric functions. While these can be eliminated via the substitution $\cos{t}= (u^2-1)/(u^2+1)$, and $\sin{t}=2u/(u^2+1)$ where $u=\cot\frac{t}{2}$, the expressions for the areas of both $K$ and $K_1$ will become complicated rational functions of $u_1,\,u_2,\,u_3$, and $u_4$. Moreover, since cotangent is decreasing on $(0,\pi)$ we have the side constraint $u_1>u_2>u_3>u_4$.

The third method was suggested by Elkies and it employs Gauss' pentagonal formula \cite{gauss}. Given a convex pentagon $K=ABCDE$, denote the areas of the peripheral triangles of $K$ by $\sigma_1=\Delta(ABC)$, $\sigma_2=\Delta(BCD)$, $\sigma_3=\Delta(CDE)$, $\sigma_4=\Delta(DEA)$, and $\sigma_5=\Delta(EAB)$.

Gauss' pentagonal formula states that $\Delta(K)$, the area of the pentagon is the larger root of the quadratic equation
\begin{equation}\label{gauss}
\Delta(K)^2-(\sigma_1+\sigma_2+\sigma_3+\sigma_4+\sigma_5)\Delta(K)+(\sigma_1\sigma_2+\sigma_2\sigma_3+\sigma_3\sigma_4+\sigma_4\sigma_5+\sigma_5\sigma_1)=0
\end{equation}
Elkies' idea is to express both $\Delta(K)$ and $\Delta(K_1)$ in terms of the variables $\sigma_i$, $1\le i\le 5$.
Of course, after an eventual scaling, one can set $\sigma_1=1$, so the problem still depends on four variables.
The appeal of this method is that one eventually wants to show that the extremal value for the ratio $\Delta(K_1)/\Delta(K)$ is attained when all the peripheral triangles have the same area, as this is a necessary and sufficient condition for the pentagon to be affine regular. Moreover, there are no side conditions to worry about as the variables $\sigma_i$ are independent of each other.

The downside of this method is that the expressions of $\Delta(K)$ and especially $\Delta(K_1)$ are rather complicated as they
contain radicals. This is just an immediate consequence of Gauss' pentagonal formula above. By the time one eliminates the radicals, the resulting expression becomes a polynomial of large degree, and again quite difficult to handle.

While all the above approaches failed to produce a proof, there is a lot to be learned from them as we have to deal with two competing goals. On one hand, convexity has to come into play at some point. This means that one way or another, there are going to be some constraints; the simpler they are the better. On the other hand, the objective function $\Delta(K_1)/\Delta(K)$ has to have a rather manageable form.

It is clear that this is an optimization problem in four variables with several side conditions, and not an easy one.
The key to solving this problem is in choosing the {\it right variables}. This is what we are going to do in the following section.
\end{section}
\begin{section}{\bf The Setup}
Throughout the remainder of the paper we use the {\it outer product} of two vectors to express areas.
This operation, also known as {\it exterior product}, or {\it wedge product} is defined as
follows.

For any two vectors $\mathbf{v}=(a, \, b)$ and $\mathbf{u}=(c,\,d)$ let the outer product of $\mathbf{v}$ and
$\mathbf{u}$ be given by
\begin{equation}\label{outerproduct}
\mathbf{v}\wedge\mathbf{u}:=(ad-bc)/2.
\end{equation}
The outer product represents the {\it signed
area} of the triangle determined by the vectors $\mathbf{v}$ and
$\mathbf{u}$, where the $\pm$ sign depends on whether the angle
between $\mathbf{v}$ and $\mathbf{u}$ - measured in the
counterclockwise direction from $\mathbf{v}$ towards $\mathbf{u}$ -
is smaller than or greater than $180^{\circ}$.

The following properties of the outer product are simple
consequences of the definition:
\begin{itemize}
\item{anti-commutativity: $\mathbf{v}\wedge\mathbf{u}=-
\mathbf{u}\wedge\mathbf{v}$ and in particular
$\mathbf{v}\wedge\mathbf{v}=0$.}

\item{linearity: $(\alpha\mathbf{v}+\beta
\mathbf{u})\wedge\mathbf{w}=\alpha\mathbf{v}\wedge\mathbf{w}
+\beta\mathbf{u}\wedge\mathbf{w}.$}
\end{itemize}

Let $K=ABCDE$ be an arbitrary convex pentagon. After an
eventual relabeling of the vertices we may assume that
\begin{equation}\label{assumption}
\Delta(ABC)=\max\{\Delta(ABC),\, \Delta(BCD),\, \Delta(CDE),\,
\Delta(DEA),\, \Delta(EAB)\},
\end{equation}
that is, we assume that $ABC$ is the peripheral triangle of largest area.

Denote the intersection of $AD$ and $CE$ by $O$. Then define
$\mathbf{v}_1 = \overrightarrow{OA}$, $\mathbf{v}_2 =
\overrightarrow{OC}$. After an appropriate scaling, we may assume
that $\mathbf{v}_1\wedge \mathbf{v}_2 = \Delta(AOC) = 1$.

Since $E$, $O$, and $C$ are collinear and $D$, $O$, and $A$ are
collinear, we can write $\overrightarrow{DO} = a \cdot
\overrightarrow{OA} = a \mathbf{v}_1$ and $\overrightarrow{EO} = b
\cdot \overrightarrow{OC} = b \mathbf{v}_2$, with $a,\ b > 0$ (see
figure \ref{l3}).

\noindent Using the triangle rule, we obtain that
$\overrightarrow{CD} = -a\mathbf{v}_1 - \mathbf{v}_2$,
$\overrightarrow{DE} = a\mathbf{v}_1 - b\mathbf{v}_2$, and
$\overrightarrow{EA} = \mathbf{v}_1 +b\mathbf{v}_2$.\\
\noindent We know that every vector in the plane can be written as a
linear combination of any two independent vectors. Set
$\overrightarrow{OB} =\mathbf{v}_3 = c \mathbf{v}_1 + d
\mathbf{v}_2$ - refer again to figure \ref{l3}.

\begin{figure}[!htb]
\centering
\includegraphics[scale=1.4]{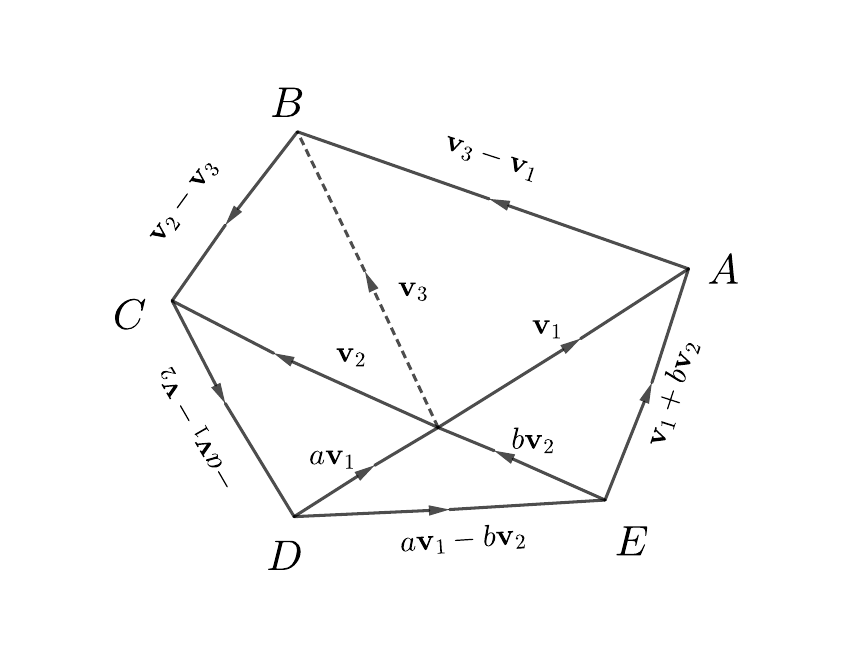}
\vspace{-1cm}
\caption{The setup for the convex pentagon problem }
\label{l3}
\end{figure}

 We also know that
$\overrightarrow{AB} = \mathbf{v}_3 - \mathbf{v}_1$ and
$\overrightarrow{BC} = \mathbf{v}_2 - \mathbf{v}_3$. We have that
\begin{eqnarray*}
\Delta(OAB) &=& \mathbf{v}_1\wedge \mathbf{v}_3 =
\mathbf{v}_1\wedge (c\mathbf{v}_1 + d\mathbf{v}_2) = d,\\
\Delta(OBC)&=& \mathbf{v}_3\wedge\mathbf{v}_2 = (c\mathbf{v}_1 +
d\mathbf{v}_2)\wedge \mathbf{v}_2 = c.
\end{eqnarray*}
After similar calculations, we can write the areas of various
triangles in the pentagon $ABCDE$ in terms of the positive constants
$a,\,b,\,c,\,d$ as shown below:
 $\Delta(OCD) =a\mathbf{v}_1\wedge\mathbf{v}_2 = a, \,\,\Delta(OEA) =
a\mathbf{v}_1\wedge b\mathbf{v}_2 =ab,\,\,\Delta(ODE)=
\mathbf{v}_1\wedge b\mathbf{v}_2 =b.$ We can now compute the total
area of the pentagon.

$\Delta(ABCDE) = \Delta(OAB) + \Delta(OBC) + \Delta(OCD) +
\Delta(ODE) + \Delta(OEA)$, that is,
\begin{equation}\label{ABCDE}
\Delta(K)=\Delta(ABCDE) = a+b+c+d+ab.
\end{equation}
\noindent Next, we compute the areas of the peripheral triangles of the pentagon.
\begin{eqnarray}\label{ears}
\Delta(ABC)&=&\overrightarrow{AB}\wedge\overrightarrow{BC}=(\mathbf{v}_3-\mathbf{v}_1)\wedge(\mathbf{v}_2-\mathbf{v}_3)=c+d-1,\notag\\
\Delta(BCD)&=&\overrightarrow{BC}\wedge\overrightarrow{CD}=(\mathbf{v}_2-\mathbf{v}_3)\wedge(-a\mathbf{v}_1-\mathbf{v}_2)=a-ad+c,\notag\\
\Delta(CDE)&=&\overrightarrow{CD}\wedge\overrightarrow{DE}=(-a\mathbf{v}_1-\mathbf{v}_2)\wedge(a\mathbf{v}_1-b\mathbf{v}_2)=ab+a,\label{ptriangles}\\
\Delta(DEA)&=&\overrightarrow{DE}\wedge\overrightarrow{EA}=(a\mathbf{v}_1-b\mathbf{v}_2)\wedge(\mathbf{v}_1+b\mathbf{v}_2)=ab+b,\notag\\
\Delta(EAB)&=&\overrightarrow{EA}\wedge\overrightarrow{AB}=(\mathbf{v}_1+b\mathbf{v}_2)\wedge(\mathbf{v}_3-\mathbf{v}_1)=b-bc+d.\notag
\end{eqnarray}

Using now assumption \eqref{assumption} we have $\Delta(ABC)\ge \Delta(BCD)$ from which $c+d-1\ge a-ad+c$ and finally after simplifying we obtain that $d\ge 1$.

Similarly, $\Delta(ABC)\ge \Delta(EAB)$ implies that $c+d-1\ge b-bc+d$, which eventually gives that $c\ge 1$.
These two simple inequalities will play a major role in the sequel. We list them below for easy future reference.

\begin{obs}
With the notations introduced above and using \eqref{assumption} it follows that
\begin{equation}\label{cd}
c\geq 1 \qquad \text{and} \qquad d\geq 1.
\end{equation}
\end{obs}

Recall the notation \eqref{periphery} introduced in the first section. It follows that
\begin{align}\label{areaears}
\Omega(K)&=\Delta(ABC)+\Delta(BCD)+\Delta(CDE)+\Delta(DEA)+\Delta(EAB)=\\
&=2(a+b+c+d+ab)-1-ad-bc\notag.
\end{align}
Let us assess the situation for a moment. At this point, we have relatively simple expressions for the
area of the pentagon $ABCDE$, and for the areas of its peripheral triangles, in terms of the variables $a$, $b$, $c$ and $d$.
Notice that while $a$, $b$, $c$ and $d$ are strictly positive they are not completely independent of each other as convexity implies that all five peripheral triangles have positive (signed) areas. The first, third and fourth quantities do already satisfy this condition as $a$ and $b$ are both positive and, as noticed earlier, $c\ge 1$ and $d\ge 1$.  It remains to require that $\Delta(BCD)$ and $\Delta(EAB)$ listed in \eqref{ears} are positive.  Hence, we have the following constraints:
\begin{equation}\label{constraints}
a-ad+c>0,\quad b-bc+d>0,\quad \text{where $a>0$, $b>0$, $c\ge 1$, and $d\ge 1$}.
\end{equation}
In some sense, one can say we are walking the middle road: the areas have quite simple expressions while the side conditions are not that complicated either. The quantities $a$, $b$, $c$, and $d$ are exactly the {\it right variables} we alluded to at the end of the previous section.

What are the values of $a$, $b$, $c$, and $d$ if $K$ is a regular pentagon? It is very easy to see that in this case $OABC$ is a parallelogram so necessarily $\mathbf{v}_3=\mathbf{v}_1+\mathbf{v}_2$, which means that $c=d=1$.
On the other hand, a simple trigonometry exercise shows that $a=b=\frac{1}{2}\sec(\pi/5)=(\sqrt{5}-1)/2$.

We record this observation for future reference.
\begin{obs}\label{extremum}
The pentagon $K$ is affine regular iff $a=b=(\sqrt{5}-1)/2$ and $c=d=1$.
\end{obs}

It remains to see how $\Delta(K_1)$ looks in terms of these variables. Before doing that we prove a result that is going to be needed later.

\begin{lemma}\label{Omegaga}
For every convex pentagon $K$ we have that $\Omega(K)<2\Delta(K)$. The constant $2$ cannot be replaced by a smaller one.
\end{lemma}
\begin{proof} Let us first prove a slightly weaker inequality. Every vertex of the pentagon $K$ is incident to exactly two non-overlapping peripheral triangles. Obviously, the sum of the areas of these two triangles is smaller than the area of the pentagon. Averaging these inequalities over all five vertices it follows that $\Omega(K)<2.5\Delta(K)$.
To see that $2.5$ can be replaced by $2$ it is enough to compare relations \eqref{ABCDE} and \eqref{areaears}. Indeed, we have
\begin{equation*}
\Omega(K)=2\Delta(K)-(1+ad+bc)<2\Delta(K).
\end{equation*}

Finally, let us show that $2$ is the best possible constant. For some $x>0$, take $a=b=x$ and $c=d=1$.
Then, conditions \eqref{constraints} are satisfied, so $K$ is a convex pentagon.
But then using \eqref{ABCDE} and \eqref{areaears} again we have that
$$\frac{\Omega(K)}{\Delta(K)}=\frac{2x^2+2x+3}{x^2+2x+2},$$
and this ratio can be arbitrarily close to $2$ if $x$ is large enough.
\end{proof}
\end{section}

\begin{section}{\bf Computing $\Delta(K_1)$}

In this section we compute the areas of the marginal triangles of the pentagon $ABCDE$. According to figure \ref{l4}, these are
$ABM$, $BCN$, $CDO$, $DEP$ and $EAQ$. Since the technique is the same in all five cases we are going to present the details just for the first marginal triangle. We are going to refer often to figure \ref{l3} from the previous section. Also recall that $\mathbf{v}_1\wedge\mathbf{v}_2=1$, and since $\mathbf{v}_3=c\mathbf{v}_1+d\mathbf{v}_2$ we also have that $\mathbf{v}_1\wedge\mathbf{v}_3=d$, and $\mathbf{v}_2\wedge\mathbf{v}_3=-c$.

Note that $\Delta(ABM)= \overrightarrow{AB}\wedge\overrightarrow{AM}$. We already know that that $\overrightarrow{AB}=\mathbf{v}_3-\mathbf{v}_1$. Since $A$, $M$ and $C$ are collinear, there exists a constant $\lambda \in (0,1)$ such that $\overrightarrow{AM}=\lambda\overrightarrow{AB}$. Similarly, since $B$, $M$ and $D$ are collinear, there exists a constant $\mu \in (0,1)$ such that $\overrightarrow{MB}=\lambda\overrightarrow{DB}$.
It follows that
\begin{equation*}
\overrightarrow{AM}+\overrightarrow{MB}=\overrightarrow{AB} \rightarrow \lambda(\mathbf{v}_2-\mathbf{v}_1)+\mu(a\mathbf{v}_1+\mathbf{v}_3)=\mathbf{v}_3-\mathbf{v}_1.
\end{equation*}
Multiplying both sides by $(a\mathbf{v}_1+\mathbf{v}_3)$ and using the properties of the outer product we obtain that
\begin{equation*}
\lambda(\mathbf{v}_2-\mathbf{v}_1)\wedge(a\mathbf{v}_1+\mathbf{v}_3)=(\mathbf{v}_3-\mathbf{v}_1)\wedge(a\mathbf{v}_1+\mathbf{v}_3),
\end{equation*}
which after expanding gives
\begin{equation*}
\lambda(-a-c-d)=-ad-d \quad \text{that is}, \quad \lambda=\frac{d(a+1)}{a+c+d}\longrightarrow \overrightarrow{AM}=\frac{d(a+1)}{a+c+d}(\mathbf{v}_2-\mathbf{v}_1).
\end{equation*}
Using now that $\Delta(ABM)= \overrightarrow{AB}\wedge\overrightarrow{AM}$ we finally obtain that
\begin{equation*}
\Delta(ABM)=(\mathbf{v}_3-\mathbf{v}_1)\wedge\frac{d(a+1)}{a+c+d}(\mathbf{v}_2-\mathbf{v}_1)=\frac{d(a+1)(c+d-1)}{a+c+d}.
\end{equation*}

The areas of the remaining four marginal triangles are computed in a similar fashion. We omit the straightforward
calculations, but list all five expressions for easy future reference.

\begin{align}\label{marginals}
\Delta(ABM)&=\frac{d(a+1)(c+d-1)}{a+c+d},\,\Delta(BCN)=\frac{c(a+c-ad)}{a+c},\,\Delta(CDO)=a,\notag\\
\Delta(DEP)&=\frac{b(ab+ad+bc)}{b+d},\,\Delta(EAQ)=\frac{(1+b)(b+d-bc)}{b+c+d}.
\end{align}
\vspace{-2cm}
\begin{figure}[!htb]
\centering
\includegraphics[scale=2.65]{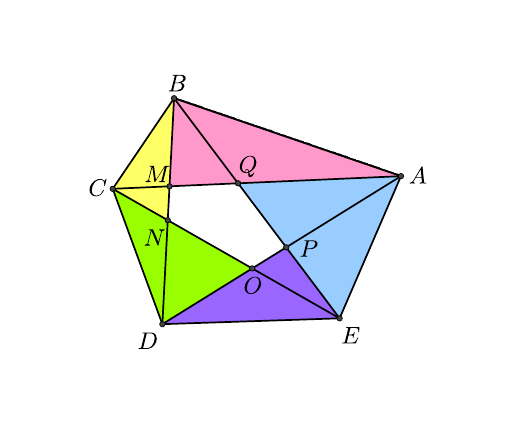}
\vspace{-1cm}
\caption{The marginal triangles}
\label{l4}
\end{figure}

Recall that $K_1=MNOPQ$ and $\Delta(K_1)=\Delta(K)-\Phi(K)$, where $\Phi(K)$
is the sum of the areas of the marginal triangles given in \eqref{marginals} above - see \eqref{periphery}.
A simple calculation shows that Theorem \ref{main} is equivalent to showing that $\Phi(K)\ge (3\sqrt{5}-5)\Delta(K)/2$.

Combining now relations \eqref{periphery}, \eqref{constraints}, \eqref{ABCDE} and \eqref{marginals}
we restate Theorem \ref{main} in an equivalent purely algebraic form.

\begin{thm} Let $a>0$, $b>0$, $c\ge 1$ and $d\ge 1$ be numbers such that $a-ad+c>0$ and $b-bc+d>0$. Then
\begin{align}\label{algebra}
\frac{d(a+1)(c+d-1)}{a+c+d}&+\frac{c(a+c-ad)}{a+c}+a+\frac{b(ab+ad+bc)}{b+d}+\\
&+\frac{(1+b)(b+d-bc)}{b+c+d}- \frac{3\sqrt{5}-5}{2}(a+b+c+d+ab)\ge 0,\notag
\end{align}
with equality if and only if $a=b=(\sqrt{5}-1)/2$ and $c=d=1$.
\end{thm}
At this juncture, it is reasonable to ask ourselves whether the inequalities $a-ad+c>0$ and $b-bc+d>0$
play any role whatsoever; it turns out that the answer is negative! This makes our job much easier since we are
dealing with an unconstrained optimization problem.

Nevertheless, the problem is nowhere close to being solved. Any attempt of proving inequality \eqref{algebra} via any of the standard methods (partial derivatives, Lagrange multipliers etc.) is destined to fail. We will have to try something else. First, we replace the variables $a$ and $b$ by two new variables, $x$ and $y$, defined by
\begin{equation}
x:=a\cdot\frac{\sqrt{5}+1}{2},\qquad y:=b\cdot\frac{\sqrt{5}+1}{2}.
\end{equation}
The reason for this substitution is that in the extremal case we have $a=b=(\sqrt{5}-1)/2$, that is, $x=y=1$ and therefore the expressions are going to be simpler.
Second, we clear the denominators of the left hand term in \eqref{algebra} and further we multiply this term by $(7+3\sqrt{5})/8$. This operation reduces the complexity of the coefficients of the resulting polynomial.

After these operations, Theorem \ref{main} becomes equivalent to the following

\begin{thm} Let $x>0$, $y>0$, $c\ge 1$, $d\ge 1$, and let $f(c,d,x,y)$ be the four-variable polynomial given below.
{\small{
\begin{align*}
&\phantom{aai}(14-6\sqrt{5})(c^2xy^3+d^2x^3y+cdx^3y+cdxy^3+cx^3y+dxy^3)+(32-12\sqrt{5})(c^2x^2y^2+d^2x^2y^2)+\\
&+(12\sqrt{5}-48)(c^2xy^2+d^2x^2y)+(12\sqrt{5}-24)(c^2x^2y+d^2xy^2)+(20\sqrt{5}-44)(cx^2y^3+dx^3y^2)+\\
&+(28-12\sqrt{5})(cxy^3+dx^3y)+(10\sqrt{5}-22)(cx^3y^2+dx^2y^3)+(48-24\sqrt{5})(cx^2y^2+dx^2y^2)+\\
&+(4+4\sqrt{5})(cxy^2+dx^2y)+(6\sqrt{5}-6)(c^3dx+cd^3y)+(6+2\sqrt{5})(c^3dy^2+cd^3x^2)+\\
&+(2\sqrt{5}-26)(c^3dy+cd^3x)+(8\sqrt{5}-4)(c^3xy^2+d^3x^2y)+(18-6\sqrt{5})(c^3xy+d^3xy)+\\
&+(6\sqrt{5}-30)(c^2d^2x+c^2d^2y)+(18-6\sqrt{5})(c^2dx^2+cd^2y^2)+(18-14\sqrt{5})(c^2dy^2+cd^2x^2)+\\
&+(8+4\sqrt{5})(c^2dy+cd^2x)+(4\sqrt{5}-8)(cdx^3+cdy^3)+(4\sqrt{5}-8)(c^2y^3+d^2x^3)+\\
&+(6+2\sqrt{5})(c^2y^2+d^2x^2)+(10\sqrt{5}-22)(x^3y^2+x^2y^3)+(6+2\sqrt{5})(c^4y^2+d^4x^2)+\\
&+(2\sqrt{5}-2)(c^4y+d^4x)+(8\sqrt{5}-16)(c^2dx^2y+cd^2xy^2)+(14\sqrt{5}-22)(c^2dxy^2+cd^2x^2y)+\\
&+(36-36\sqrt{5})(c^2dxy+cd^2x^2y)+(6-2\sqrt{5})(c^3dxy+cd^3xy)+(28\sqrt{5}-80)(cdx^2y+cdxy^2)+\\
&+12(c^3d^2+c^2d^3)+4(c^4d+cd^4)-8\sqrt{5}(c^3y^2+d^3x^2)+4x^2y^2+(18+6\sqrt{5})cdxy+\\
&+(12-4\sqrt{5})c^2d^2xy+(36-16\sqrt{5})x^3y^3+(70-30\sqrt{5})cdx^2y^2.
\end{align*}}}
Then $f(c,d,x,y)\ge 0$ with equality if and only if $c=d=x=y=1$.
\end{thm}
The first reaction when seeing $f(c,d,x,y)$ is that it is a really horrible expression. At a second thought,
it is not that bad. After all, $f(c,d,x,y)$ is a polynomial with $73$ terms whose monomials have degrees 
$4$, $5$, and $6$. What makes it look worse are the coefficients since they contain a lot of $\sqrt{5}$'s. 
However, there is some symmetry since $f(c,d,x,y)=f(d,c,y,x)$.
This is because the areas does not change if you flip the pentagon. In the end, showing that $f(c,d,x,y)\ge 0$
does not seem such an impossible task.

How exactly are we going to achieve this? The idea is to write $f(c,d,x,y)$ as a sum of nonnegative terms, all of which vanish when $c=d=x=y=1$.

In particular, we will be able to express $f(c,d,x,y)$ as a sum of terms of the following form
\begin{align}\label{fasasum}
f(c,d,x,y)=\sum_{J} &p_J\,x^{i_1}y^{i_2}(c-1)^{i_3}(d-1)^{i_4}(x-1)^{2i_5}(y-1)^{2i_6}(x-c)^{2i_7}(y-d)^{2i_8}\cdot \notag\\
&\cdot(x-d)^{2i_9}(y-c)^{2i_{10}}(c-d)^{2i_{11}}(x-y)^{2i_{12}}(xy-1)^{2i_{13}}(xy-cd)^{2i_{14}},
\end{align}
with coefficients $p_J\ge 0$ for every choice of $J=(i_1,i_2,i_3,i_4,i_5,i_6,i_7,i_8,i_9,i_{10},i_{11},i_{12},i_{13},i_{14})$.

Here $i_k$ are nonnegative integers for all $1\le k\le 14$. Moreover, for any choice of $J$ we will have $i_5+i_6+\ldots+i_{13}+i_{14}\le 1$, that is,
at most one of the last 10 factors will appear in any particular term. As a result, we were able to write $f(c,d,x,y)$ as follows

\begin{align}\label{Qs}
f(c,d,x,y)=(x-1)^2\,Q_5&+(y-1)^2\,Q_6+(x-c)^2\,Q_7+(y-d)^2\,Q_8+(x-d)^2\,Q_9+(y-c)^2\,Q_{10}+\notag\\
&+(c-d)^2\,Q_{11}+(x-y)^2\,Q_{12}+(xy-1)^2\,Q_{13}+(xy-cd)^2\,Q_{14}+Q_0,
\end{align}
where $Q_0, Q_5, Q_6,\ldots Q_{14}$ are polynomials in $c-1$, $d-1$, $x$ and $y$ all whose coefficients are positive.
These polynomials were obtained by identifying the coefficients of similar monomials on both sides of \eqref{fasasum}
and then locating a nonnegative solution of the resulting linear system. We used the symbolic algebra software Maple to perform
these calculations.

Moreover, since $f(c,d,x,y)=f(d,c,y,x)$ we were able to find symmetric solutions which satisfy $Q_5(c,d,x,y)=Q_6(d,c,y,x)$, $Q_7(c,d,x,y)=Q_8(d,c,y,x)$, and $Q_9(c,d,x,y)=Q_{10}(d,c,y,x)$. In addition, $Q_k(c,d,x,y)=Q_k(d,c,y,x)$ for $k\in \{0, 11,12,13,14\}$. We list these polynomials below.

{\small
\begin{align*}
Q_5=&4y+(53-23\sqrt{5})xy+(98-36\sqrt{5})(d-1)y+(72-31\sqrt{5})(d-1)xy+(75-24\sqrt{5})(d-1)^2+\\
&+(8+8\sqrt{5})(d-1)^2y+(4\sqrt{5}-8)(d-1)^2x+(67-27\sqrt{5})(d-1)^3+(2+2\sqrt{5})(d-1)^3y+\\
&+4(d-1)^4+(36\sqrt{5}-78)(c-1)y+(28\sqrt{5}-62)(c-1)(d-1)+(21-\sqrt{5})(c-1)(d-1)^2+\\
&+(24\sqrt{5}-48)(c-1)(d-1)^2y+(12-4\sqrt{5})(c-1)(d-1)^3+(12-4\sqrt{5})(c-1)^2(d-1)y.
\\
Q_6=&4x+(53-23\sqrt{5})xy+(98-36\sqrt{5})(c-1)x+(72-31\sqrt{5})(c-1)yx+(75-24\sqrt{5})(c-1)^2+\\
&+(8+8\sqrt{5})(c-1)^2x+(4\sqrt{5}-8)(c-1)^2y+(67-27\sqrt{5})(c-1)^3+(2+2\sqrt{5})(c-1)^3x+\\
&4(c-1)^4+(36\sqrt{5}-78)(d-1)x+(28\sqrt{5}-62)(c-1)(d-1)+(21-\sqrt{5})(d-1)(c-1)^2+\\
&+(24\sqrt{5}-48)(d-1)(c-1)^2x+(12-4\sqrt{5})(d-1)(c-1)^3+(12-4\sqrt{5})(d-1)^2(c-1)x.
\end{align*}}

{\small
\begin{align*}
Q_7=&(65-26\sqrt{5})+(8\sqrt{5}-16)x+(40\sqrt{5}-88)xy^2+(108-40\sqrt{5})(d-1)+(12\sqrt{5}-24)(d-1)x+\\
&+(20\sqrt{5}-44)(d-1)xy^2+(14-6\sqrt{5})(d-1)^2xy+(3\sqrt{5}-3)(d-1)^3+12(c-1)+\\
&+(4\sqrt{5}-8)(c-1)x+(30-13\sqrt{5})(c-1)xy+(10\sqrt{5}-22)(c-1)xy^2+(7\sqrt{5}-5)(c-1)(d-1)\\
&+(4\sqrt{5}-8)(c-1)(d-1)x+(14-6\sqrt{5})(c-1)(d-1)xy+(15-3\sqrt{5})(c-1)(d-1)^2+\\
&+(24\sqrt{5}-50)(c-1)^2y.
\\
Q_8=&(65-26\sqrt{5})+(8\sqrt{5}-16)y+(40\sqrt{5}-88)yx^2+(108-40\sqrt{5})(c-1)+(12\sqrt{5}-24)(c-1)y+\\
&+(20\sqrt{5}-44)(c-1)yx^2+(14-6\sqrt{5})(c-1)^2yx+(3\sqrt{5}-3)(c-1)^3+12(d-1)+\\
&+(4\sqrt{5}-8)(d-1)y+(30-13\sqrt{5})(d-1)yx+(10\sqrt{5}-22)(d-1)yx^2+(7\sqrt{5}-5)(c-1)(d-1)+\\
&+(4\sqrt{5}-8)(d-1)(c-1)y+(14-6\sqrt{5})(c-1)(d-1)xy+(15-3\sqrt{5})(d-1)(c-1)^2+\\
&+(24\sqrt{5}-50)(d-1)^2x.
\\
Q_9&=(20\sqrt{5}-43)+(24\sqrt{5}-48)(d-1)+(31-9\sqrt{5})(c-1)(d-1)+4(c-1)(d-1)^2.
\\
Q_{10}&=(20\sqrt{5}-43)+(24\sqrt{5}-48)(c-1)+(31-9\sqrt{5})(d-1)(c-1)+4(d-1)(c-1)^2.
\\
Q_{11}&=(2+2\sqrt{5})(x^2+y^2)+(32\sqrt{5}-38)xy+(70-30\sqrt{5})(x+y)xy++(4\sqrt{5}-8)(c-1)(d-1)xy+\\
&+(8\sqrt{5}-12)x^2y^2+(50-14\sqrt{5})(c-1+d-1)xy+(2+2\sqrt{5})((c-1)y^2+(d-1)x^2).
\\
Q_{12}&=(17-7\sqrt{5})xy+(\sqrt{5}-2)(c-1+d-1)xy+(66-24\sqrt{5})(c-1)(d-1).
\\
Q_{13}&=(30\sqrt{5}-59)+(36-16\sqrt{5})xy+(64\sqrt{5}-130)(c-1+d-1)+(66\sqrt{5}-130)(c-1)(d-1).
\\
Q_{14}&=47-18\sqrt{5}.
\\
Q_0&=(71\sqrt{5}-152)((c-1)^3+(d-1)^3)xy+(2+2\sqrt{5})((c-1)^4y^2+(d-1)^4x^2)+\\
&+(16-4\sqrt{5})((c-1)^3y+(d-1)^3x)xy+(2+2\sqrt{5})(c-1)(d-1)((c-1)^2x+(d-1)^2y)+\\
&+(6\sqrt{5}-6)(c-1)(d-1)((c-1)^2y^2+(d-1)^2x^2)+(56-22\sqrt{5})((c-1)^4y+(d-1)^4x)+\\
&+(24\sqrt{5}-36)((c-1)^3y^2+(d-1)^3x^2)+(194-86\sqrt{5})((c-1)^3y+(d-1)^3x)+\\
&+(12\sqrt{5}-15)((c-1)^2y^2+(d-1)^2x^2)+(84\sqrt{5}-162)(c-1)(d-1)(c-1+d-1)+\\
&+(98-42\sqrt{5})(c-1)(d-1)xy((c-1)y+(d-1)x)+(100-40\sqrt{5})(c-1)(d-1)xy(x+y)+\\
&+(30\sqrt{5}-35)(c-1)^2(d-1)^2+(16\sqrt{5}-32)(c-1)^2(d-1)^2xy+\\
&+(114\sqrt{5}-242)(c-1)(d-1)+(108-42\sqrt{5})(c-1)(d-1)xy.
\end{align*}}
It is straightforward to check that all the coefficients appearing in the polynomials $Q_k$ above are positive.
Since $x, y, c-1$ and $d-1$ are all nonnegative it follows that $f(c,d,x,y)\ge 0$ as soon as equality \eqref{fasasum} holds. 
Moreover, $f(c,d,x,y)=0$ if and only if $x=y=0$ and $c=d=1$, that is, by observation \ref{extremum}, if and only if the pentagon $K=ABCDE$ is an affine regular pentagon.
This completes the proof of the main theorem.
\end{section}

\newpage
\begin{section}{\bf A generalization attempt}

Given $K=A_1A_2\ldots A_n$ a convex polygon with $n$ vertices, $n\ge 3$, and a fixed $0<r\le 1$,
consider the point $B_k$ on the edge $A_{k+1}A_{k+2}$ such that $A_{k+1}B_k/A_{k+1}A_{k+2}=r$.
Consider the $n$-gon, $K_r$, bounded by the segments $A_kB_k$ - see figure \ref{l5}. Until now we studied the case
$r=1$.
\vspace{-1cm}
\begin{figure}[!htb]
\centering
\includegraphics[scale=2]{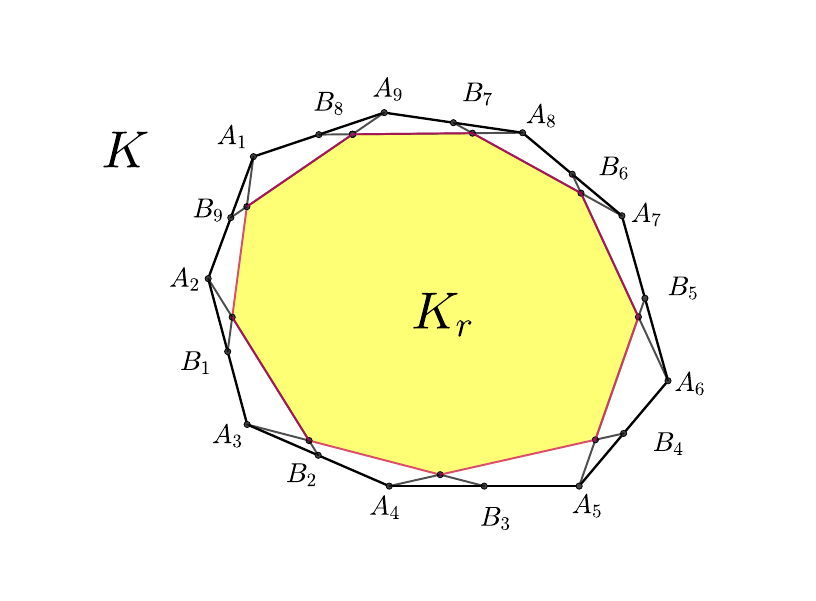}
\vspace{-1cm}
\caption{The polygon $K_r$}
\label{l5}
\end{figure}

The question again is, for a fixed $n$, to find the range of the ratio $\Delta(K_r)/\Delta(K)$, where $K$ is an $n$-sided convex polygon.

If $K$ is a triangle, it is known that
\begin{equation}
\frac{\Delta(K_r)}{\Delta(K)}=\frac{(2r-1)^2}{r^2-r+1}.
\end{equation}

This can be derived from a more general formula of Routh \cite{routh} who proved that if three cevians divide the side of the triangle $T$ in the ratio $r$, $s$ and $t$, they bound a triangle $T'$ whose area is given by
\begin{equation}
\frac{\Delta(T')}{\Delta(T)}=\frac{(rst-(1-r)(1-s)(1-t))^2}{(1-s+rs)(1-t+st)(1-r+rt)}.
\end{equation}
If $K$ is a quadrilateral, Ash et al. proved in \cite{aaa} that
\begin{equation}
\frac{(1-r)^3}{r^2-r+1}< \frac{\Delta(K_r)}{\Delta(K)}\le \frac{(1-r)^2}{r^2+1},
\end{equation}
and this is optimal. In their proof, they use cartesian coordinates and fix three of the vertices of $K$ at $(0,0)$, $(1,0)$ and $(0,1)$.
We have a simpler proof that uses the outer product approach.

Not surprisingly, for general $r$ nothing is known for $n\ge 5$. We make two observations.

One one hand, it is easy to see that $\Delta(K_r)\ge \Delta(K_1)$ since the former polygon contains the latter.
Hence, the best upper bound for $\Delta(K_r)/\Delta(K)$ is the trivial one.

On the other hand, the lower bound situation is more interesting as for small enough $r$, the ratio $\Delta(K_r)/\Delta(K)$ is bounded away from $0$. We have the following result
\begin{thm}\label{lastthm}
For all $n\ge 5$ and every $0<r<1/2$ we have that
\begin{equation}\label{hhh}
\frac{\Delta(K_r)}{\Delta(K)}\ge 1-2r.
\end{equation}
\end{thm}
\begin{proof}
From figure \ref{l5} we can see that
\begin{equation}
\Delta(K_r)+r\Omega(K)=\Delta(K_r)+\sum\Delta(A_kA_{k+1}B_k)>\Delta(K).
\end{equation}
But in Lemma \ref{Omegaga} we showed that $\Omega(K)<2\Delta(K)$ if $K$ is a pentagon;
the inequality holds for all convex $n$-gons. For instance, if $K$ is a hexagon is it
clear that the set of peripheral triangles can be partitioned into two classes, such
that no two triangles within the same class have interior common points.
This implies that $\Delta(K_r)+2r\Delta(K)>\Delta(K)$, from which \eqref{hhh} follows.
\end{proof}

The problem of determining the range of $\Delta(K_r)/\Delta(K)$ for fixed $r$ in $(0,1)$
and $K$ a pentagon is probably a very difficult one. One additional reason for expecting this is that numerical experiments suggest
that if $r<0.88$ then the maximum of $\Delta(K_r)/\Delta(K)$ is no longer reached for the case of the affine regular pentagon.

\end{section}

\thispagestyle{empty}
{\footnotesize{
}}

\end{document}